\documentclass[11pt,a4paper]{article}

\usepackage{amsmath,amsfonts,amsthm,enumerate}
\usepackage{url}
\usepackage[a4paper, margin=35mm]{geometry}

\newtheorem{nn}{}
\newtheorem{thm}[nn]{Theorem}
\newtheorem{lem}[nn]{Lemma}
\newtheorem{prop}[nn]{Proposition}
\newtheorem{cor}[nn]{Corollary}
\newtheorem{rem}[nn]{Remark}
\newtheorem{dfn}[nn]{Definition}
\newtheorem*{claim*}{Claim}

\newcommand{\setcond}[2]{\left\{ #1\,:\, #2\right\}}
\newcommand{\sprod}[2]{\left< #1\, , \, #2\right>}

\newcommand{\N}{\mathbb{N}}
\newcommand{\Z}{\mathbb{Z}}
\newcommand{\R}{\mathbb{R}}

\newcommand{\im}{\operatorname{im}}

\newcommand{\cl}{\operatorname{cl}}
\newcommand{\cS}{\mathcal{S}}
\newcommand{\sxc}{\operatorname{sxc}}
\newcommand{\sxd}{\operatorname{sxdeg}}
\newcommand{\cA}{\mathcal{A}}
\newcommand{\cone}{\operatorname{cone}}

\newcommand{\CP}{\operatorname{CP}}

\newcommand{\ux}{\mathbf{x}}
\newcommand{\SDR}{\operatorname{SDR}}

\begin{document}
	
\title{\huge Optimal size of linear matrix inequalities in semidefinite approaches to polynomial optimization}
\author{ Gennadiy~Averkov\footnote{Institute of Mathematical Optimization, Faculty of Mathematics, Otto-von-Guericke-Universit\"at Magdeburg, email: averkov@ovgu.de}}

\maketitle

\begin{abstract}
	The abbreviations LMI and SOS stand for `linear matrix inequality' and `sum of squares', respectively.
	The cone $\Sigma_{n,2d}$ of SOS polynomials in $n$ variables of degree at most $2d$ is known to have a semidefinite extended formulation with one LMI of size $\binom{n+d}{n}$. In other words, $\Sigma_{n,2d}$ is a linear image of a set described by one LMI of size $\binom{n+d}{n}$. We show that $\Sigma_{n,2d}$ has no semidefinite extended formulation with finitely many LMIs of size less than $\binom{n+d}{n}$. Thus, the standard extended formulation of $\Sigma_{n,2d}$ is optimal in terms of the size of the LMIs. As a direct consequence, it follows that the cone of $k \times k$ symmetric positive semidefinite matrices has no extended formulation with finitely many LMIs of size less than $k$. We also derive analogous results for further cones considered in polynomial optimization such as truncated quadratic modules, the cones of copositive and completely positive matrices and the cone of sums of non-negative circuit polynomials. 
\end{abstract}

\section{Introduction}

\subsection{Semidefinite extended formulations}

Consider the vector space $\cS^k$ of $k \times k$ symmetric matrices over $\R$ and the cone $\cS_+^k$ of positive semidefinite matrices in $\cS^k$. If $ A: \R^n \to \cS^k$ is an affine map, say 
\[
	A(x_1,\ldots,x_n):=A_0 + x_1 A_1 + \cdots + x_n A_n,
\]
with $A_0,\ldots,A_n \in \cS^k$, then the condition \[
	A(x_1,\ldots,x_n) \in \cS_+^k
\] is called a \emph{linear matrix inequality} (\emph{LMI}) of size $k$ on real-valued variables $x_1,\ldots,x_n$.  \emph{Semidefinite programming} (\emph{SDP}) is optimization of a linear function subject to finitely many LMIs \cite{Handbook:semidefinite:2000,Handbook:SDP:CONIC:POP:12}.  Equivalently, SDP can also be described as optimization of a linear function over the intersection of an affine subspace $H$ of $\cS^k$ with the cone $\cS_+^k$. Due to the stunning expressive power of LMIs, SDP has numerous applications across a wide range of subject areas \cite{Handbook:semidefinite:2000}. 

While SDP is known to be efficiently solvable -- with a desired accuracy  -- under mild assumptions, the size of the LMIs is definitely an important limitation on the way to practical solvability \cite{Mittelmann:2003}. In order to successfully use SDP solvers, it is thus important to keep the size of the respective LMIs under control when modeling an underlying problem. The aim of this article is to address this size issue from the theoretical viewpoint. We are interested in understanding the limitation on the expressive power of the SDP implied by prescribing a size bound on the underlying LMIs. More concretely, we discuss semidefinite relaxations of problems in polynomial optimization. 

Our aim is to study properties of the so-called semidefinite extended formulations of semialgebraic sets.
We will use the general conic-programming framework from Gouveia~et~al.~\cite{GPT:2013} that allows to deal with various types of conic extended formulations in a uniform fashion. If $K$ is a closed convex cone in a finite-dimensional $\R$-vector space and $S= \pi(K \cap H)$, where $H$ is an affine space and $\pi$ is a linear map, then we say that $S$ has a \emph{$K$-lift}. For 
\[
	K= (\cS_+^k)^m = \underbrace{\cS_+^k \times \cdots \times \cS_+^k}_{m},
\] 
a set $S$ having a $K$-lift is a linear image of a set that can be described by $m$ LMIs of size $k$. In this case, we also say that $S$ has a \emph{semidefinite extended formulation} with $m$ LMIs of size $k$. 

\begin{dfn}[Semidefinite extension complexity and semidefinite extension degree] Let $S$ be a subset of an $\R$-vector space.
We call the minimal $k$ such that $S$ has a $\cS_+^k$-lift the \emph{semidefinite extension complexity} of $S$ and denote this value by $\sxc(S)$. If $S$ has no $\cS_+^k$-lift independently of the choice of $k$, we define $\sxc(S):=\infty$. As a natural complement to $\sxc(S)$, we introduce the \emph{semidefinite extension degree} $\sxd(S)$ of $S$ to be the smallest $k$ such that $S$ has an $(\cS_+^k)^m$-lift for some finite $m$. If $S$ has no semidefinite extended formulation, we define $\sxd(S):=\infty$. 
\end{dfn}

Studying lower and upper bounds on $\sxc(S)$ is an active research area \cite{GPT:2013,MR3450076,MR3395545,MR3421630,MR3397065,MR3555385,MR3369076,MR3388236,GGS:2017,MR3652002,AKW:2018,MR3771398}. It is clear that $\sxd(S) \le \sxc(S)$. We believe that, along with $\sxc(S)$, the value $\sxd(S)$ is an important parameter for quantifying tractability of semidefinite approaches to optimization of linear functions over $S$. 

\subsection{Convex cones in polynomial optimization}

We briefly revise some basic concepts and facts from polynomial optimization, see also  \cite{Marshall:book:2008,Laurent:2009,Lasserre:book:2015}. 

In what follows, let $m,n$ and $k$ be positive integers and $d$ a non-negative integer. Let $\ux=(x_1,\ldots,x_n)$ and let $\R[\ux]$ be the ring of $n$-variate polynomials in variables $x_1,\ldots,x_n$ with coefficients in $\R$.  The subset
\[
\R[\ux]_d := \setcond{f \in \R[\ux]}{\deg f \le d}.
\]
of $\R[\ux]$ is a vector space of dimension $\binom{n+d}{n}$. A polynomial $f \in \R[\ux]$ is called \emph{sum of squares} (\emph{SOS}) if $f=f_1^2 + \cdots + f_r^2$ holds for finitely many polynomials $f_1,\ldots,f_r \in \R[\ux]$.

The following are the basic cones from real algebraic geometry and polynomial optimization: 
\begin{align*}
	\Sigma_{n,2d} & := \setcond{f \in \R[\ux]_{2d}}{f \ \text{is SOS}},
	\\ P_{n,2d} & := \setcond{f \in \R[\ux]_{2d}}{f \ge 0 \ \text{on} \ \R^n},
	\\ P_{n,2d}(X) &:= \setcond{f \in \R[\ux]_{2d}}{f \ge 0 \ \text{on} \ X} & & (X \subseteq \R^n).
\end{align*}

It is well-known \cite[\S2.1]{Lasserre:book:2015} that $\Sigma_{n,2d}$ has the semidefinite extended formulation
\begin{align}
	\label{sos:lift}
	\Sigma_{n,2d} & =\setcond{v_{n,d}^\top A v_{n,d}}{A \in \cS_+^k} & & \text{for} \ k=\binom{n+d}{n},
\end{align}
where $v_{n,d}$ is the vector
\begin{align}
	\label{vec:mon}
	v_{n,d} := (\ux^\alpha)_{|\alpha| \le d}
\end{align}
of all monomials of degree at most $d$ in $n$ variables and the notation $\ux^\alpha$ with $\alpha = (\alpha_1,\ldots,\alpha_n) \in \Z_+^n$ is used to denote the monomial
\[
\ux^\alpha := x_1^{\alpha_1} \cdots x_n^{\alpha_n}
\] of degree 
\[
|\alpha| := \alpha_1 + \cdots + \alpha_n.
\]
Equality \eqref{sos:lift} implies
\begin{align}
	\sxd(\Sigma_{n,2d}) \le \sxc(\Sigma_{n,2 d}) \le \binom{n+d}{n}.
	\label{Sigma:bounds}
\end{align}
The lifted representation \eqref{sos:lift} of $\Sigma_{n,2d}$ is a basic building block for the reduction of polynomial optimization problems to SDP problems. Due to the obvious inclusion $\Sigma_{n,2d} \subseteq P_{n,2d}$, lower bounds on the unconstrained polynomial optimization problem
\begin{align}
	\label{upop}
	& \inf_{x \in \R^n} f(x) & & (f \in \R[\ux]_{2d}),
\end{align}
can be derived from the \emph{SOS relaxation} of \eqref{upop}, which is the conic problem -- with respect to the cone $\Sigma_{n,2d}$ -- formulated as 
\begin{align}
	\label{usos}
	\max \setcond{\lambda \in \R}{f- \lambda \in \Sigma_{n,2d}}.
\end{align}
In view of \eqref{sos:lift},  the condition $f - \lambda \in \Sigma_{n,2d}$ in \eqref{usos} can be reformulated as the linear constraint $\lambda + v_{n,2d}^\top A v_{n,d} =f$ on the scalar decision variable $\lambda \in \R$ and the matrix decision variable $A \in \cS_+^k$ of size $k = \binom{n+d}{d}$. Reformulated like this, \eqref{usos} becomes a semidefinite optimization problem.

 For a general constrained polynomial optimization problem 
 \begin{align}
	 \label{cpop}
	 \inf \setcond{f(x)}{x \in \R^n, g_1(x) \ge 0,\ldots, g_k(x) \ge 0} & & (f, g_1,\ldots,g_k \in \R[\ux])
 \end{align}
 the approach is similar. Feasible solutions of \eqref{cpop} form the set
 \[
 X := \setcond{x \in \R^n}{g_1(x) \ge 0,\ldots,g_k(x) \ge 0}.
 \]
 Choosing $d$ with $2d \ge \deg f$, one can reformulate \eqref{cpop} as the conic problem with respect to the cone $P_{n,2d}(X)$: 
 \[ \inf \setcond{ \lambda \in \R}{f - \lambda \in P_{n,2d}(X)}.
 \]
 In constrained polynomial optimization, the principle of the SOS-based approaches is to find a cone $C$ contained in $P_{n,2d}(X)$ that approximates $P_{n,2d}(X)$ sufficiently well and has a semidefinite extended formulation. Real algebraic geometry suggests various natural choices of $C$ that are built upon $\Sigma_{n,2d}$. The so-called \emph{SOS hierarchies} for \eqref{cpop} involve cones of the form 
 \begin{align}
	 \label{C:truncated:module}
	 C= \Sigma_{n,2d_0} + g_1 \Sigma_{n, 2d_1}  + \cdots + g_k \Sigma_{n,2d_k}
 \end{align}
 with $d_0,\ldots,d_k \in \Z_+$ \cite[\S2.4.2 and \S2.7.1]{Lasserre:book:2015}. We call the cone $C$ given by \eqref{C:truncated:module} the \emph{truncated quadratic module} generated by $g_1,\ldots,g_s$ with the \emph{truncation degrees} $2d_0,\ldots,2d_k$. The standard approach is to first choose the value $d_0 \in \Z_+$ such that $2d _0$ is an upper bound on the degrees of the of the polynomials $f, g_1,\ldots,g_k$ and then to fix the largest possible values $d_1,\ldots,d_k \in \Z_+$ satisfying $2 d_i + \deg g_i \le 2 d_0$. The truncated modules with the above special choice of the truncation degrees generate the so-called \emph{SOS hierarchy}, while the choice of $d_0$ determines the \emph{level} of the hierarchy.
 
 \subsection{Overview of results} 
 
 We address the following basic questions:
 
 \begin{enumerate}[(Q1)]
	\item \label{first:question} How large is $\sxd(C)$ for closed convex cones $C$ satisfying $\Sigma_{n,2d} \subseteq C \subseteq P_{n,2d}(X)$? 
	\item \label{second:question} How large is $\sxd(C)$ for $C$ being a truncated quadratic module? 
 \end{enumerate}
 
 Our main theorem (Theorem~\ref{thm:tool}) suggests an approach to lower-bounding $\sxd(C)$ for the above cases. Using this approach, we can answer (Q\ref{first:question}) and (Q\ref{second:question}) in a variety of cases. Regarding (Q\ref{first:question}), it should be mentioned that recent breakthrough results of Scheiderer \cite{Scheiderer:2018} provide various choices of convex semi-algebraic sets $C$, for which $\sxd(C)$ is infinite. For example, $\sxd(P_{n,2d})$ is infinite if $n,d \ge 2$ and $(n,d) \ne (2,2)$ \cite[Corollary~4.25]{Scheiderer:2018}. Our quantitative studies are in a certain sense complementary, because our objective is to determine $\sxd(C)$ in those cases, for which this value is finite.  A recent contribution of Fawzi \cite{Fawzi:2018} can be interpreted as a first step in the study of quantitative aspects of (Q\ref{first:question}). Arguments of Fawzi allow to determine $\sxd(\cS_+^k)$ for $k \le 3$ and $\sxd(\Sigma_{n,2d})$ for $n=1$ and $d \le 2$. Regarding $\sxd(\Sigma_{n,2d})$ in the case $n=1$ and $d \le 2$, see also the exposition in \cite{ahmadi2017improving}. 
 
 In this paper, we determine $\sxd(\cS_+^k)$ and $\sxd(\Sigma_{n,2d})$ for all $k$, $n$ and $d$. We also determine the semidefinite extension degree of the truncated quadratic modules under a natural assumption. 
 
 Apart from SOS-based approaches, there has been a new approach to polynomial optimization based on the so-called SONC cone $C_{n,2d}$, considered in the work of Dressler et al.~\cite{dressler2016approach,MR3691721}. It has not been clear if this alternative approach has a semidefinite formulation. It turns out that this is indeed the case. Moreover, $C_{n,2d}$ even has a second-order cone extended formulation. This can be expressed as the equality $\sxd(C_{n,2d})=2$ in our notation. 
 
The paper is organized as follows. In Section~\ref{sect:results}, we formulate and discuss the results. Section~\ref{sect:background} provides background information, including the notation and two basic tools that we need for proving our main theorem (Theorem~\ref{thm:tool}). Section~\ref{sect:proof:main} contains the proof of the main theorem. Section~\ref{sect:proof:conseq:main} presents proofs of the consequences of the main theorem.  Section~\ref{sec:SONC} deals with the SONC cone.
 
\section{Results}

\label{sect:results}

\subsection{Main theorem and its consequences}

Lower bounds on the semidefinite extension degree of various specific convex cones that we discuss below will be obtained as a consequence of the following general result. 

\begin{thm}[Main theorem]
	\label{thm:tool}
	Let $X \subseteq \R^n$ be a set with non-empty interior. Let $C \subseteq P_{n,2d}(X)$ be a closed convex cone such that there exist finite subsets $S$ of $X$ of arbitrarily large cardinality with the following property:
	\begin{itemize}
		\item[$(\ast)$]For every $k$-element subset $T$ of $S$, some  polynomial $f$ in the cone $C$ is equal to zero on $T$ and is strictly positive on $S \setminus T$.
	\end{itemize}
	Then $\sxd(C) > k$. 
\end{thm}

Specializing Theorem~\ref{thm:tool} to more concrete situations we obtain a number of corollaries. Their detailed proofs are given in Section~\ref{sect:proof:conseq:main}. As mentioned in the introduction, SOS-based approaches to polynomial optimization use conic formulations based on cones $C$ that lie between $\Sigma_{n,2d}$ and $P_{n,2d}(X)$. In view of Theorem~\ref{thm:tool}, the semidefinite extension degree of such cones $C$ is necessarily `large'. In the case $\Sigma_{n,2d} \subseteq C \subseteq P_{n,2d}(X)$, choosing $f$ in $(\ast)$ to be appropriate polynomials from $\Sigma_{n,2d}$, we arrive at
\begin{cor}
	\label{cor:between}
	Let $X \subseteq \R^n$ be a set with non-empty interior and $C$ be a closed convex cone satisfying $\Sigma_{n,2d} \subseteq C \subseteq P_{n,2d}(X)$. Then 
	\[
	\sxd(C) \ge \binom{n+d}{n}.
	\] 
\end{cor}

Since $\Sigma_{n,2d}$  has a semidefinite extended formulation with one LMI of size $\binom{n+d}{d}$, the latter corollary allows to determine the exact values of the semidefinite extension degree and the semidefinite extension complexity for the cone $\Sigma_{n,2d}$:

\begin{cor}
	\label{cor:sxd:sos}
	\(\sxd(\Sigma_{n,2d})=\sxc(\Sigma_{n,2d})=\binom{n+d}{n}\).
\end{cor}

Corollary~\ref{cor:sxd:sos} allows to determine the computational costs of solving the SOS-relaxation \eqref{usos} of an unconstrained polynomial optimization problem by means of SDP. 

Turning to constrained polynomial optimization, we determine the semidefinite extension degree of the truncated quadratic modules, which allows us to estimate the costs of solving a given level of the SOS hierarchy. The following corollary deals with the natural case when the set $X$ of feasible solutions of the underlying optimization problem has non-empty interior. 

\begin{cor}
	\label{cor:sxd:trunc:quad:module}
	Let $g_1,\ldots,g_k \in \R[\ux] \setminus \{0\}$ be such that the set
	\[
		X := \setcond{x\in \R^n}{g_1(x) \ge 0,\ldots,g_k(x) \ge 0}
	\]
	has non-empty interior. Then, for the truncated quadratic module
	\[
		C= \Sigma_{n,2d_0} + g_1 \Sigma_{n,2 d_1} + \cdots + g_k \Sigma_{n,2 d_k}
	\]
	with $d_0,\ldots,d_k \in \Z_+$, one has \[\sxd(C) = \binom{n+d}{n},\] where \[d := \max \{d_0,\ldots,d_k\}.\]
\end{cor}

The cone $C$ in Corollary~\ref{cor:sxd:trunc:quad:module} has a straightforward semidefinite extended formulation with $k+1$ LMIs of sizes $\binom{n+d_0}{n}, \ldots, \binom{n+d_k}{n}$. The number $\binom{n+d}{n}$ is the maximum of these $k+1$ sizes. By Corollary~\ref{cor:sxd:trunc:quad:module}, the straightforward extended formulation is optimal in terms of the size of the LMIs when $X$ has non-empty interior. 

The case $d=1$ of Corollary~\ref{cor:sxd:sos} yields the semidefinite extension degree of $\cS_+^k$: 

\begin{cor}
	\label{cor:sxd:sdp}
	\(\sxd(\cS_+^k) = k\).
\end{cor}

Corollary~\ref{cor:sxd:sdp} implies that the expressive power of the semidefinite optimization grows strictly with the growth of the size $k$ of the underlying LMIs. In other words, the family of all convex semialgebraic sets that have a semidefinite extended formulation (we call such sets \emph{semidefinitely representable}) can be decomposed into the hierarchy of the families 
\[
	\SDR(k):=\setcond{S \subseteq \R^n}{n \in \N, \ \sxd(S) \le k}
\] 
with each level of the hierarchy being strictly larger than the previous one. The lowest level $\SDR(1)$ of the hierarchy is just the family of all polyhedra. The family $\SDR(1)$ corresponds to linear optimization. The next level $\SDR(2)$ corresponds to the second-order cone programming,  which is an important generalization of linear programming. The family $\SDR(2)$ can be characterized using the \emph{second-order cone}
\[
	L_m:=\setcond{(x_1,\ldots,x_m) \in \R^m}{x_m \ge \sqrt{x_1^2 + \cdots +  x_{m-1}^2}}
\]
(for $m=1$, we define $L_m:=\R_+$).

\begin{prop}[Folklore; see the discussion in \cite{Fawzi:2018}]
	\label{prop:soc}
	For $S \subseteq \R^n$, the following conditions are equivalent
	\begin{enumerate}[(i)]
		\item $\sxd(S) \le 2$.
		\item $S$ has an $(L_3)^m$-lift for some $m \in \N$.
		\item $S$ has an $L_m$-lift for some $m \in \N$.  
	\end{enumerate}
\end{prop}

Our results cover the following recent results as a special case. Aiming to demonstrate the discrepancy between the expressive power of second-order cone programming and semidefinite programming,  Fawzi proved
\begin{thm}[Fawzi \cite{Fawzi:2018}]
	\label{thm:fawzi}
	$\sxd(\cS_+^3) = 3$.
\end{thm}

In \cite{Fawzi:2018}, Fawzi also explains how to determine $\sxd(\Sigma_{1,4})$. Note that the cone $\Sigma_{1,4}$ is not discussed in the arxiv version (arxiv:1610.04901) of Fawzi's paper \cite{Fawzi:2018}. Independently, Ahmadi et al. \cite[Theorem~5]{ahmadi2017improving} refer to the arxiv version of \cite{Fawzi:2018} and provide a short argument that allows to determine $\sxd(\Sigma_{1,4})$ by reusing Fawzi's proof of Theorem~\ref{thm:fawzi}.

\begin{thm}[Ahmadi et al. {\cite[Thm.~5]{ahmadi2017improving}}, Fawzi {\cite[Sec.~4]{Fawzi:2018}}]
	\label{thm:ahmadi:et:al}
	$\sxd(\Sigma_{1,4}) = 3$.
\end{thm}

The notation in \cite{Fawzi:2018} and \cite{ahmadi2017improving} is different, but results from these sources have a straightforward interpretation as a derivation of the equalities $\sxd(\cS_+^3) = 3$ and $\sxd(\Sigma_{1,4}) = 3$. Theorems~\ref{thm:fawzi}  and \ref{thm:ahmadi:et:al} are special cases of our Corollaries~\ref{cor:sxd:sdp} and \ref{cor:sxd:sos}, respectively.

The proof of the lower bound $\sxd(\cS_+^3) \ge 3$ of Fawzi is based on the idea that a special face-incidence structure of the convex cone $\cS_+^3$ is an obstruction to having an $(\cS_+^2)^m$-lift with a small $m$. Combinatorial obstructions to having an $\R_+^m$-lift are thoroughly studied in linear and discrete optimization  \cite{FKPT:2013}, but for semidefinite optimization, the respective theory is not as developed yet, and Fawzi's contribution is a first step in this new direction. Since $\cS_+^3$ is a non-polyhedral cone, its face lattice is infinite. The relevant face incidences of a given closed convex set $S$ can be extracted from the so-called slack matrix of $S$. Loosely speaking, the slack matrix provides results $f(s)$ of evaluation of all linear inequalities $f \ge 0$ valid for $S$ at all points $s$ of $S$. Recently, Gouveia et al.~\cite{GPT:2013} developed a criterion for testing if a given convex set has a $K$-lift for convex cones $K$ under some mild assumptions on $K$. Arguing by contradiction, Fawzi assumes the existence of $(\cS_+^2)^m$-lift for $\cS_+^3$. He then applies the slack-matrix criterion from \cite{GPT:2013} for $K=(\cS_+^2)^m$ and gives a purely combinatorial graph-theoretic argument, which yields a lower bound on $m$. This lower bound on $m$ depends on the number face-incidences of $\cS_+^3$ taken into account and can be made arbitrarily large by choosing sufficiently many incidences.

Our proof approach to Theorem~\ref{thm:tool} is inspired by the arguments of Fawzi \cite{Fawzi:2018}. Following his ideas, we also rely on the slack-matrix criterion provided in \cite{GPT:2013}. The combinatorial argument from \cite{Fawzi:2018} can be replaced by a direct application of Ramsey's theorem for graphs.  To prove Theorem~\ref{thm:tool}, we use Ramsey's theorem for $k$-uniform hypergraphs, with the  case $k=2$ corresponding to graphs.

Essentially, our proof of Corollary~\ref{cor:sxd:sos} is a generalization of the proof idea from \cite[Section~IV-B]{ahmadi2017improving}. The extension for $n=1$ and an arbitrary $d$ from the case $n=1,d=2$ considered in \cite[Section~IV-B]{ahmadi2017improving} is rather straightforward, but for passing from $n=1$ to an arbitrary $n$, somewhat more work is needed. 

In view of Corollary~\ref{cor:sxd:sos}, to determine $\sxd(P_{n,2d})$, it suffices to combine a classical result of Hilbert with a recent result of Scheiderer. 

\begin{thm}[Hilbert \cite{Hilbert:1888}]
	\label{thm:hilbert}
	\begin{align*}
		& \Sigma_{n,2d} \ne P_{n,2d}  & & \iff  & & n, d \ge 2, \ (n,d) \ne (2,2).
	\end{align*}
\end{thm}

\begin{thm}[Scheiderer {\cite[Corollary~4.25]{Scheiderer:2018}}]
	\label{thm:scheiderer}
	If $n \ge 2, d \ge 2$ and $(n,d) \ne (2,2)$, then $P_{n,2d}$ has no semidefinite extended formulation.
\end{thm}

Directly combining Corollary~\ref{cor:sxd:sos}, Theorem~\ref{thm:hilbert} and Theorem~\ref{thm:scheiderer}, we get

\begin{cor}
	\label{cor:nonneg:cone}
	\[
		\sxd(P_{n,2d})= \sxc(P_{n,2d})=
		\begin{cases} 
			\infty, & \text{if} \ n, d \ge 2, \ (n,d) \ne (2,2).
			\\ \binom{n+d}{n}, & \text{otherwise}.
		\end{cases}
	\]
\end{cor}

Results from \cite{GPT:2013} imply that $\sxc(C)$ and $\sxd(C)$ are invariant under duality of cones: if $C \subseteq \R^n$ is a $n$-dimensional pointed closed convex cone, then 
\begin{align}
	\sxc(C) & =\sxc(C^\ast), \label{inv:dual:sxc}	
	\\  \sxd(C) & =\sxd(C^\ast), \label{inv:dual:sxd}
\end{align}
where $C^\ast$ is the dual cone of $C$.  Via dualization, Corollaries~\ref{cor:between} and \ref{cor:sxd:sos} yield a number of consequences. We introduce the \emph{moment cones}
\begin{align*}
	M_{n,2d} &:= \cl ( \cone (\setcond{v_{n,2d}(x)}{x \in \R^n})),
	\\ M_{n,2d}(X) & := \cl ( \cone ( \setcond{v_{n,2d}(x)}{x \in X}) & & (X \subseteq \R^n),
\end{align*}  
where $\cl$ stands for the Euclidean topological closure and $\cone$ for the convex conic hull. Representability of the moment cones via $\cS_+^k$-lifts has been studied by Scheiderer: 

\begin{thm}[Scheiderer {\cite[Corollary~4.24]{Scheiderer:2018}}]
	\label{thm:sch:mom:cones}
	Let $X \subseteq \R^n$ be a semi-algebraic set with non-empty interior and let $n, d \ge 2$ and $(n,d) \ne (2,2)$. Then $M_{n,2d}(X)$ has no semidefinite extended formulation. In particular, $M_{n,2d}$ has no semidefinite extended formulation, too. 
\end{thm}

If $n=1$ or $d=1$ or $(n,d)=(2,2)$, Theorem~\ref{thm:sch:mom:cones} does not rule out the possibility of $\sxd(M_{n,2d}(X))$ being finite. In these cases, the following consequence of  Corollary~\ref{cor:between} and \eqref{inv:dual:sxd} can be used to provide lower bounds on $\sxd(M_{n,2d}(X))$: 

\begin{cor}
	\label{cor:mom:cones:X}
	For every $X \subseteq \R^n$ with non-empty interior,
	\[ \sxd(M_{n,2d}(X)) \ge \binom{n+d}{n}. \]
\end{cor}

As a direct consequence of Corollary~\ref{cor:between}, \eqref{inv:dual:sxd} and Theorem~\ref{thm:sch:mom:cones} we also obtain the exact values of $M_{n,2d}$ for all $n$ and $d$:
\begin{cor}
	\label{cor:mom:cones}
	\[
	\sxd(M_{n,2d})= \sxc(M_{n,2d}) =
	\begin{cases} 
	\infty, & \text{if} \ n, d \ge 2, \ (n,d) \ne (2,2),
	\\ \binom{n+d}{n}, & \text{otherwise}.
	\end{cases}
	\]	
\end{cor}

The cone 
\[
	\CP_k := \setcond{A \in \cS^k}{x^\top A x \ge 0  \ \text{for all} \ x \in \R_+^k}
\]
is known as the cone of \emph{copositive matrices} of size $k$. Its dual cone $\CP_k^\ast$ is the closed convex cone generated by rank-one positive semidefinite matrices $x x^\top$ with $x \in \R_+^n$. Elements of $\CP_k^\ast$ are called \emph{completely positive} matrices. Note that various well-known hard combinatorial optimization problems can be modeled as conic optimization problems with respect to the cones $\CP_k$ and $\CP_k^\ast$ \cite{Duer:2010}. We provide a lower bound the semidefinite extension degree of both cones and determine it exactly for small values of $k$:

\begin{cor}
	\label{cor:copositive}
	One has 
	$\sxd(\CP_k) = \sxd(\CP_k^\ast) \ge k$, and the equality $\sxd(\CP_k) = \sxd(\CP_k^\ast)= k$ holds if $k  \le 4$. 
\end{cor}

The exact values $\sxd(\CP_k)$ for $k > 4$ are left undetermined. In fact, it is not even known if these values are finite \cite[\S5.2]{Scheiderer:2018}.

From the above results we draw the conclusion that standard SOS-based approaches to polynomial optimization \emph{necessarily} lead to semidefinite problems with large LMIs, which are usually hard to solve in practice. A solution to this issue could be to use sparsity or symmetry of underlying problems, if applicable; see \cite[Ch.~8]{Lasserre:book:2015}, \cite{MR3395550} and \cite{ahmadi2017improving}. Alternatively, one can look for new ways of reduction of polynomial optimization problems to convex problems. There are a number of results in this direction \cite{MR0214374,MR2727311,MR2968862,MR3499559,MR3691721,dressler2016approach,Dressler:Diss,MR3592772}. 

\subsection{Results for the SONC cone}

An alternative approach to polynomial optimization, suggested by Dressler et al. in \cite{MR3691721,dressler2016approach,Dressler:Diss}, is based on the cone $C_{n,2d}$ of \emph{sums of non-negative circuit polynomials} (abbreviated as \emph{SONC polynomials}) in $n$ variables of degree at most $2d$. As reported in \cite{dressler2016approach}, the optimization approach  based on $C_{n,2d}$ leads to convex problems that frequently can be solved efficiently in practice. Furthermore, this alternative approach seems to be not as sensitive to the choice of $n$ and $d$ as the well-known approach based on $\Sigma_{n,2d}$ \cite{seidler2018experimental}.

We transform the original definition of $C_{n,2d}$, which is given Section~\ref{sec:SONC}, to a less technical definition more suitable for our purposes. For a finite set $A \subseteq \Z_+^n$, we first introduce the cone 
\begin{equation}
\label{PnA:eq}
P_{n,A}:= \setcond{f = \sum_{\alpha \in A} f_\alpha \ux^\alpha}{f \ge 0 \ \text{on} \ \R^n}.
\end{equation}
of non-negative polynomials $f$ with the support $\setcond{\alpha}{f_\alpha \ne 0}$ of $f$ contained in $A$.
While for an arbitrary $A$, it is hard to find an explicit description of $P_{n,A}$ in terms of the coefficients $f_\alpha$ of the polynomial $f$, there are special cases of `sparse' sets $A$, in which such a description is known. If the convex hull of $A$ is a $k$-dimensional simplex with vertices $\alpha(0),\ldots,\alpha(k)$ belonging to $(2\Z_+)^n$ and $A$ consists of the $k+1$ vertices of this simplex and another point $\beta$ that lies in the relative interior of the simplex, then $P_{n,A}$ has a simple inequality description derivable from the weighted arithmetic-geometric mean inequality. We denote by $\cA_{n,2d}$ the family of all $A$ having the special form as above and satisfying the inclusion $A \subseteq \setcond{\alpha \in \Z_+^n}{|\alpha| \le 2d}$. The inclusion for $A$ ensures  $P_{n,A} \subseteq P_{n,2d}$ for every $A \in \cA_{n,2d}$ so that one has $C_{n,2d} \subseteq P_{n,2d}$. 

It turns out that the SONC cone $C_{n,2d}$ can be represented as the sum 
\[
	C_{n,2d} = \sum_{A \in \cA_{n,2d}} P_{n,A}.
\] 

 The following theorem provides theoretical support to the informal message that $C_{n,2d}$ is `practically tractable'. The smallest possible semidefinite extension degree for a non-polyhedral cone is two. The following result shows that the semidefinite extension degree of $C_{n,2d}$ is that small, independently of the choice of $n$ and $d$. A well-known result of Ben-Tal and Nemirovski \cite[\S\S2.3.5]{MR1857264} shows that the hypograph of a weighted geometric-mean function with rational weights has a second-order cone extended formulation. Directly applying this fact to the explicit inequality description of the cones $\sxd(P_{n,A})=2$ occurring in the above description of $C_{n,2d}$ we obtain 

\begin{thm}
	\label{sxd:SONC}
	For all $n,d \in \N$, one has 
	$\sxd(C_{n,2d}) = 2$.
\end{thm}

By Proposition~\ref{prop:soc} and Theorem~\ref{sxd:SONC}, the \emph{SONC relaxation} 
\begin{equation}
	\label{sonc:rexation}
	\inf \setcond{\lambda \in \R}{f - \lambda \in C_{n,2d}}
\end{equation}
of the unconstrained polynomial optimization problem \eqref{upop} can be formulated as a second-order cone problem. While $\sxd(C_{n,2d})$ remains the same for all $n$ and $d$, the cone $C_{n,2d}$ does become more complex with the growth of $n$ and $d$. It would also be interesting to study $\sxc(C_{n,2d})$ and to determine the number of constraints needed in a second-order cone or semidefinite extended formulation of $C_{n,2d}$. Such studies would shed light on how to formulate \eqref{sonc:rexation} compactly in the paradigms of semidefinite and second-order programming. 

In \cite[Prop.~7.2]{MR3481195} it was shown that $C_{n,2d}$  and $\Sigma_{n,2d}$ are not comparable with respect to inclusion for $n, d \ge 2, \ (n,d) \ne (2,2)$. So, the cone $\Sigma_{n,2d} + C_{n,2d}$ is strictly larger than both $\Sigma_{n,2d}$ and $C_{n,2d}$ for these choices of $(n,d)$, which implies that the \emph{SOS+SONC relaxation} 
\begin{equation}
	\label{SOS+SONC:rel}
	\inf \setcond{\lambda \in \R}{f - \lambda \in \Sigma_{n,2d} + C_{n,2d}}
\end{equation}
of \eqref{upop} is a stronger relaxation than both the SOS relaxation and the SONC relaxation.

The following corollary addresses the question on the relation between $P_{n,2d}$ and $\Sigma_{n,2d}+C_{n,2d}$, which was formulated in the PhD thesis \cite[p.~134]{Dressler:Diss} of Mareike~Dressler and asked by Raman~Sanyal during the defense of her thesis. Above results imply that the cone $\Sigma_{n,2d}+C_{n,2d}$  always has a semidefinite extended formulation, while in view of Scheiderer's result, the cone $P_{n,2d}$ has a semidefinite extended formulation only in the cases of equality $\Sigma_{n,2d} = P_{n,2d}$, which was characterized by Hilbert. This yields 

\begin{cor}
	\label{cor:SOS+SONC}
	For all $n,d \in \N$, one has:
	\begin{align*}
		\Sigma_{n,2d} + C_{n,2d} & \ne P_{n,2d} & & \iff & & n,d \ge 2, \ (n,d) \ne (2,2).
	\end{align*}
\end{cor}

Corollary~\ref{cor:SOS+SONC} shows that, in general,  \eqref{SOS+SONC:rel} is not equivalent to the original problem \eqref{upop}.

\section{Background material}

\label{sect:background}

\subsection{Basic notation and terminology}

Let $\N:=\{1,2,3,\ldots\}$ be the set of all positive integers. We use $\Z_+$ resp. $\R_+$ to denote the set of non-negative integer resp. real values. Let $[k]:=\{1,\ldots,k\}$ for $k \in \N$ and $[k]:=\emptyset$ for $k=0$. The cardinality of a set $X$ is denoted by $|X|$. Given a set $X$ and a non-negative integer $k \ge 0$, we denote by $\binom{X}{k}$ the set of all $k$-element subsets of $X$. If $X_1,\ldots,X_r$ are finitely many sets in a vector space, the  sum of $X_1,\ldots,X_r$ is introduced as 
\[
	X_1 + \cdots  + X_r := \setcond{u_1 + \cdots + u_r}{u_1 \in X_1,\ldots,u_r \in X_r}.
\] 
If $g  \in \R$ and $C \subseteq \R^n$ or $g \in \R[\ux]$ and $C \subseteq \R[\ux]$, we use the notation 
\[
	g C := \setcond{g p}{p \in C}. 
\]

If $A$ is a matrix, then $A^\top$ denotes the transpose of $A$. Vectors are interpreted as columns in matrix expressions. The image of a matrix $A \in \R^{m  \times n}$ is $\im(A) := \setcond{A x }{x \in \R^n}$.

\subsection{Euclidean spaces and convex sets}

We endow the space $\R^n$ with the standard scalar product $\sprod{x}{y}:=x^\top y$. Two linear subspaces $X$ of $Y$ of $\R^n$ are said to be \emph{orthogonal} if $\sprod{x}{y}=0$ holds for all $x \in X$ and $y \in Y$. In the space $\cS^k$ of $k \times k$ symmetric matrices over $\R$ we introduce the scalar product of $A= (a_{ij})_{i,j \in [k]}$ and $B = (b_{ij})_{i,j \in [k]}$ by 
\[
	\sprod{A}{B} := 
\sum_{i,j \in [k]} a_{ij} b_{ij}.
\]
The set $\cS_+^k$ is the convex closed cone of positive semidefinite matrices in $\cS^k$.
For $x,y \in \R^k$ the rank-one symmetric matrices $x x^\top, y y^\top \in \cS_+^k$ satisfy the relation 
\begin{equation}
	\label{sprod:rank:1}
	\sprod{x x^\top}{y y^\top}  = \sprod{x}{y}^2.
\end{equation}

We endow the space $(\cS^k)^m$ of $m$-tuples of $k \times k$ symmetric matrices  with the scalar product 
\[
	\sprod{A}{B} := \sum_{i=1}^m \sprod{A_i}{B_i}
\]
for $A=(A_1,\ldots,A_m), B=(B_1,\ldots,B_m) \in (\cS^k)^m$. 

By $\cone$ we denote the convex conic hull and by $\cl$ the topological closure with respect to the Euclidean topology.
For a non-empty set $X \subseteq \R^n$, we define the \emph{conic dual} \cite{Rockafellar:1997} of $X$ by 
\[
	X^\ast := \setcond{y}{\sprod{x}{y} \ge 0 \ \text{for all} \ x \in X}.
\]
It is well-known that 
\begin{equation}
	\label{double:dual}
	(X^\ast)^\ast = \cl(\cone(X))
\end{equation}
holds for all $X$ with $\emptyset \ne X \subseteq \R^n$. 
The conic dual is introduced in the same way in an arbitrary Euclidean space, in particular, in $(\cS^k)^m$. It is known that $\cS_+^k$ is self-dual, that is, $(\cS_+^k)^\ast = \cS_+^k$ \cite[Thm.~A.7.6]{MR1857264}. This implies that $(\cS_+^k)^m$ is self-dual, too. 

We call a convex cone $C$ in $\R^n$ \emph{pointed} if there exists $u \in \R^n \setminus \{0\}$ with $\sprod{u}{x} \ge 0$ for all $x \in C$ and such that $\setcond{x \in C}{\sprod{u}{x}=0} = \{0\}$. If $C$ is pointed, and $u$ a vector as above, then $\setcond{x \in C}{\sprod{u}{x}=1}$ is a bounded affine slice of $C$. 
The cones $\Sigma_{n,2d}, P_{n,2d}$ and $P_{n,2d}(X)$, with $X \subseteq \R^n$ having non-empty interior, are known to be pointed, closed and full-dimensional within the vector space $\R[\ux]_{2d}$. 

\subsection{Tools}

The following result is contained, albeit in somewhat different wording, in \cite{GPT:2013}. 

\begin{thm}[Gouveia et al. \cite{GPT:2013}]
	\label{gpt:thm}
	Let $C \subseteq \R^n$ be a closed convex cone and let $K=(\cS_+^k)^m$. Then the following conditions are equivalent: 
	\begin{enumerate}[(i)]
		\item \label{K:lift} $C$ has a $K$-lift.
		\item \label{factorization:general} 
		For every $x \in C$ there exist $A_x \in K$  and for every $y \in C^\ast$ there exists $B_y \in K$ such that the equality
		\[
			\sprod{x}{y} = \sprod{A_x}{B_y}
		\] holds for all $x \in C$ and $y \in C^\ast$. 
	\end{enumerate}
\end{thm}

\begin{rem}
	One can rephrase the results about $K$-lifts of $n$-dimensional compact convex sets from \cite{GPT:2013} as results about $K$-lifts of $n$-dimensional pointed closed convex cones. With this interpretation, it can be seen that Theorem~\ref{gpt:thm} in the case of $n$-dimensional pointed convex cones is covered by Remark~2.3,  Theorem~2.4 and Corollary~2.6 in \cite{GPT:2013}. Note also that Corollary~2.6 is about $K$-lifts of so-called \emph{nice} cones $K$. The cone $K= (\cS_+^k)^m$  in Theorem~\ref{gpt:thm} is nice, because  $\cS_+^k$ is known to be a nice cone; see the comment following Corollary~2.6 in \cite{GPT:2013}.
	
	We also sketch how to deduce the assertion of Theorem~\ref{gpt:thm} for general convex cones from the case of $n$-dimensional pointed convex cones. If $C$ is not full-dimensional or not pointed, then passing to appropriate coordinates, we can assume $C = C_0 \times \R^s \times \{0\}^t$, where the cone $C_0$ is $(n-s-t)$-dimensional and pointed. It is not hard to see that $C$ has a $K$-lift if and only if $C_0$ has a $K$-lift. One has $C^\ast = C_0^\ast \times \{0\}^s \times \R^{t}$. Thus, each scalar product $\sprod{x}{y}$ with $x \in C$ and $y \in C^\ast$, occurring in \eqref{factorization:general}, can be written as $\sprod{x}{y}=\sprod{x_0}{y_0}$ where $x_0 \in C_0$ is a vector of the first $n-s-t$ components of $x$ and $y_0 \in C_0^\ast$ is a vector of the first $n-s-t$ components of $y$. Conversely, each $x_0 \in C_0$ and $y_0 \in C_0$ yield vectors $x \in C$ and $y \in C^\ast$ with $\sprod{x_0}{y_0} = \sprod{x}{y}$. This shows that condition \eqref{factorization:general} holds for the cone $C$ if and only if it holds for the cone $C_0$. Thus, equivalence \eqref{K:lift} $\iff$ \eqref{factorization:general} for the general cone $C$ is derived from the same equivalence for the cone $C_0$. 
\end{rem}

\begin{rem}
	Equalities \eqref{inv:dual:sxc} and \eqref{inv:dual:sxd} from the introduction follow from Theorem~\ref{gpt:thm}.
\end{rem}

Ramsey's theory is another powerful tool that we use in our proofs. Quoting T.~S.~Motzkin, one can describe Ramsey type results as assertions that the \emph{``complete disorder is impossible''} \cite[Ch.~2]{Graham:Rothshild:Spencer:1990:Ramsey}. The most well-known version of Ramsey-type theorems is concerned with edge colorings of complete graphs. It can be illustrated by the following example. If we color each edge of the complete graph on six nodes with one of the two given colors, then -- no matter how we choose the coloring -- our colored graph will always contain a monochromatic triangle. This observation is a special case of the Ramsey theorem for graphs: if, for given $c \in \N$ and $n \in \N$, the edges of the complete graph on $N$ nodes are colored with $c$ colors then the graph will contain a complete monochromatic subgraph on $n$ nodes whenever $N$ is large enough. For the above example, the `input' of the Ramsey theorem is the number $c=2$ of colors and the size $n=3$ of the ordered substructure that we want to discover in the overall structure, while $N=6$ is a possible `output', giving the size of the overall structure sufficient for guaranteeing the existence of an ordered substructure of the desired size. In our arguments, we will need the Ramsey theorem for hypergraphs, which is a natural generalization of Ramsey's theorem for graphs from the case of the complete graph $\binom{[N]}{2}$ to the case of the complete $k$-uniform hyperpgraph $\binom{[N]}{k}$:


 	\begin{thm}[Ramsey's theorem for hypergraphs \cite{Graham:Rothshild:Spencer:1990:Ramsey}]
 		\label{thm:ramsey}
 		Let $k, n, c \in \N$. Then there  exists $N \in \N$ with the following property: for every map $F: \binom{[N]}{k} \to [c]$  there exists a subset $W$ of $[N]$ of cardinality $n$ such that $F$ is constant on $\binom{W}{k}$. 
 	\end{thm}
 
	Theorem~\ref{thm:ramsey} is a generalization of Ramsey's theorem for graphs from $k=2$ to an arbitrary $k$.
 	We denote the minimal $N$ in Theorem~\ref{thm:ramsey} by $R_k(n;c)$. The value $R_k(n;c)$ is the so-called \emph{Ramsey number} for $k$-uniform hypergraphs. In the context of Theorem~\ref{thm:ramsey}, $F(T)$ is usually called the \emph{color} assigned to $T \in \binom{[N]}{k}$.


 \section{Proof of Theorem~\ref{thm:tool}}
 
 \label{sect:proof:main}

 We give an outline of the proof of Theorem~\ref{thm:tool}. In what follows, we will use
 
  \begin{rem}
  	\label{rem:identificiation}
  	The vector spaces $\R[\ux]_{d}$ and $\R^{\binom{n+d}{n}}$ can be identified via the identification of $f = \sum_{|\alpha| \le d} f_\alpha \ux^\alpha$ with the vector $(f_\alpha)_{|\alpha| \le d}$. This allows us to write the evaluation $f(x)$ of $f$ at $x \in \R^n$ as the scalar product $f(x) = \sprod{f}{v_{n,d}(x)}$, using the vector $v_{n,d}$ of monomials defined by \eqref{vec:mon}, and to consider dual cones $C^\ast \subseteq \R^{\binom{n+d}{n}}$ of closed convex cones $C \subseteq \R[\ux]_{d}$. 
  \end{rem}

To highlight the relevant combinatorial convex geometry exploited in the proof of Theorem~\ref{thm:tool}, we introduce the following notion. We call a subset $S$ of a $n$-dimensional closed convex set $C \subseteq \R^n$ a \emph{$k$-neighborly configuration} in $C$ if, for each $k$-element subset $T$ of $S$, there exists a supporting hyperplane $H$ of $C$ satisfying $H \cap S = T$. This notion is strongly related to the well-known notion of $k$-neighborly polytope from polyhedral combinatorics. We recall that a \emph{$k$-neighborly polytope} is a polytope with every set of $k$ or fewer vertices forming a face \cite{Ziegler:Polytopes} . 
  
  Remark~\ref{rem:identificiation} implies that for the cone $C \subseteq P_{n,2d}(X)$, the vector $v_{n,2d}(s)$ belongs to the dual cone $C^\ast$ for every $s \in X$. Consequently, if $C$ had a $(\cS_+^k)^m$-lift, then by Theorem~\ref{gpt:thm}, one would have $f(s)= \sprod{f}{v_{n,2d}(s)} = \sprod{A_f}{B_t}$ for all $f \in C$ and all $s \in X$ with appropriate choices of $A_f \in (\cS_+^k)^m, B_s \in (\cS_+^k)^m$. Our proof uses the zero-pattern in the decomposition $f(s) = \sprod{A_f}{B_s}$. This means that we only make use of the distinction between $f(s)=0$ and $f(s)>0$. A set $S \subseteq X$ satisfying $(\ast)$ gives rise to polynomials $f_T \in C$ with $T \in \binom{S}{k}$ that satisfy $f_T(s)=0$ for $s \in T$ and $f_T(s) >0$ for $ s\in S \setminus T$. Interpreting $f_T(s)$ as $\sprod{f_T}{v_{n,2d}(s)}$, the latter implies that set $\setcond{v_{n,2d}(s)}{s \in S}$ is a $k$-neighborly point configuration of the cone $C^\ast$.
  
  The crucial step of the proof is to show that, if $C$ has a $(\cS_+^k)^m$-lift, then $C^\ast$ cannot contain a $k$-neighborly configuration of an arbitrarily large cardinality. The proof does not explicitly use the notion of $k$-neighborly configuration, but this notation allows to `visualize' what is happening in convex-geometric terms.  The derivation of an upper bound on the size of a $k$-neighborly configuration is carried out in Lemma~\ref{lem:key} in a somewhat more abstract setting that uses only $A_T:=A_{f_T}$ for $T \in \binom{S}{k}$ and $B_t$ with $t \in S$ and  does not explicitly involve the cone $C$. Note that for deciding whether $\sprod{A_T}{B_t}=0$ is fulfilled, one does not need the whole knowledge of $A_T$ and $B_t$. If $A_T = (A_{T,1},\ldots,A_{T,m})$ and $B_t=(B_{t,1},\ldots,B_{t,m})$, then it is easy to check that condition $\sprod{A_T}{B_t} =0$ holds if and only if $\im(A_{T,i})$ is orthogonal to the image of $\im(B_{t,i})$ for every $i \in [m]$ (see Lemma~\ref{lem:images}). Thus, the only information we need about $A_{T,i}$ and $B_{t,i}$ is the knowledge of their images. 

  The condition $f_T(s)=0$ implies that $\im(A_{T,i})$ is orthogonal to the space $U_{T,i} :=\sum_{t \in T} \im(B_{t,i})$. Thus, $U_{T,i}$ is an obstruction to the choice of $\im(A_{f,i})$. The tuple  $(U_{T,1},\ldots,U_{T,m})$ of $m$ vector spaces is a label of the hyperedge $T \in \binom{S}{k}$ in the hypergraph $\binom{S}{k}$. As we do not have any particular information about the labeling $T \in \binom{S}{k} \mapsto (U_{T,1},\ldots,U_{T,m})$, it `looks' like a completely unordered structure to us. The lack of the order can be eliminated by applying Ramsey's theorem and passing to an ordered substructure. Using Ramsey's theorem for hypergraphs (Theorem~\ref{thm:ramsey}), we are able to show that, if $S$ is large enough, then within some smaller hyperpgraph $\binom{W}{k}$ on a set $W \subseteq S$ of $k+1$ nodes,  the hyperedge $T$  has the same label $(U_1,\ldots,U_m)$ for every $T \in \binom{W}{k}$. Once the existence of $W$ as above is established, one can easily derive a contradiction and conclude the proof. 
 
 We now proceed with detailed arguments.
 
 \begin{lem}
 	\label{lem:sum:v:sp}
 	Let $U_1,\ldots,U_n$ be linear subspaces of $\R^k$ and let $U:= \sum_{i=1}^n U_i$. Then there exists a subset $I \subseteq [n]$ with $|I| \le k$ such that $U = \sum_{i \in I} U_i$. 
 \end{lem}
 \begin{proof}
 	Pick a basis $b_1,\ldots,b_m$ of $U$ from the set $U_1 \cup \ldots \cup U_k$ and then fix for each $i \in [m]$ an index $j_i \in [n]$ with $b_i \in U_{j_i}$. It is clear that the assertion holds for $I= \{j_1,\ldots,j_m \}$.
 \end{proof}

 \begin{lem}
 	\label{lem:images}
 	Let $A, B \in \cS_+^k$. Then $\sprod{A}{B} = 0$ holds if and only if $\im(A)$ is orthogonal to $\im(B)$. 
 \end{lem}
 \begin{proof}
 	Let $a_1,\ldots,a_k$ be a basis of $\R^k$ consisting of eigenvectors of $A$ corresponding to the eigenvalues $\lambda_1 \ge \ldots \ge \lambda_r > \lambda_{r+1} = \ldots = \lambda_k = 0$. Analogously, let $b_1,\ldots,b_k$ be a basis of $\R^k$ consisting of eigenvectors of $B$ corresponding to the eigenvalues $\mu_1 \ge \ldots \ge \mu_s > \mu_{s+1} = \ldots = \mu_k=0$. Then $\im(A)$ is linearly spanned by $a_1,\ldots,a_r$ and $\im(B)$ is linearly spanned by $b_1,\ldots,b_s$. Furthermore, one has 
 	\begin{align*}
 	A & = \sum_{i=1}^r \lambda_i a_i a_i^\top
 	& & \text{and} & B & = \sum_{j=1}^s \mu_j b_j b_j^\top.
 	\end{align*}
 	In view of \eqref{sprod:rank:1}, the latter representations imply
 	\[
 	\sprod{A}{B} = \sum_{i=1}^r \sum_{j=1}^s \lambda_i \mu_j \sprod{a_i}{b_j}^2.
 	\]
 	Consequently, $\sprod{A}{B} = 0$ holds if and only if $\sprod{a_i}{b_j}=0$ holds for all $i \in [r]$ and all $j \in [s]$. This gives the assertion.
 \end{proof}

  \begin{lem}[Key lemma]
  	\label{lem:key}
	Let $S$ be a set of cardinality at least $k$ and let $A_T \in (\cS_+^k)^m$, with $T \in \binom{S}{k},$ and $B_s \in (\cS_+^k)^m$, with $s \in S$, be such that the condition $\sprod{A_T}{B_s} = 0$ holds if and only if $s \in T$. 
	Then 
	\[
		|S| < R_k(k+1;(k+1)^m).
	\]
  \end{lem}
  \begin{proof}
  	The proof relies on Theorem~\ref{thm:ramsey}. Let
  	\begin{align*}
	  	A_T & = (A_{T,1},\ldots,A_{T,m}) & & \text{for} \ T \in \binom{S}{k},
	  	\\ B_s & = (B_{s,1},\ldots,B_{s,m}) & & \text{for} \ s \in S.
  	\end{align*}
  	For $T \in \binom{S}{k}$ and $i \in [m]$, consider
  	\begin{align}
	  	\label{U:def}
	  	U_{T,i}:= \sum_{t \in T} \im(B_{t,i})
  	\end{align}
  	and let 
  	\[
	  	d_{T,i}:=\dim(U_{T,i}).
	\] 
	We view  $\{0,\ldots,k\}^m$ as a set of $(k+1)^m$ colors and assign color $(d_{T,1},\ldots,d_{T,m}) \in \{0,\ldots,k\}^m$ to the set $T \in \binom{S}{k}$. 
  	
  	Assuming $|S| \ge R_k(k+1;(k+1)^m)$, we will arrive at a contradiction. By the definition of $R_k(k+1,(k+1)^m)$, there exists a subset $W$ of $S$ of cardinality $k+1$ such that all elements of $\binom{W}{k}$ are colored with the same color. This means, for some color $(d_1,\ldots,d_k) \in \{0,\ldots,k\}^m$, one has $d_{T,i}=d_i$ for all $i \in [m]$ and $T \in \binom{W}{k}$. Thus, when $T \in \binom{W}{k}$, the dimension of the vector space $U_{T,i}$  does not depend on $T$. We show that for $T \in \binom{W}{k}$, the vector space $U_{T,i}$ itself is independent of $T \in \binom{W}{k}$. This means, we will verify the following 

  	\begin{claim*} For some vector spaces $U_1,\ldots,U_m \subseteq \R^k$, one has $U_{T,i} = U_i$ for all $i \in [k]$ and $T \in \binom{W}{k}$.
  	\end{claim*}

  	If the claim was false, then we had $U_{T_1,i} \ne U_{T_2,i}$ for some $i \in [k]$ and $T_1, T_2 \in \binom{W}{k}$. Hence, $U_{T_1,i}$ and $U_{T_2,i}$ are proper subspaces of $U_{T_1,i} + U_{T_2,i}$ so that $\dim(U_{T_1,i} + U_{T_2,i}) > d_i$ holds. By \eqref{U:def}, 
  	\begin{equation}
	  	\label{sum:two:Us}
	  	U_{T_1,i} + U_{T_2,i} = \sum_{t \in T_1 \cup T_2} \im(B_{t,i}).
  	\end{equation}
  	Lemma~\ref{lem:sum:v:sp}, applied to the right-hand side of \eqref{sum:two:Us}, yields the existence of $T' \subseteq T_1 \cup T_2$ with $|T'| \le k$ such that 
  	\[
	  	U_{T_1,i} + U_{T_2,i} = \sum_{t \in T'} \im(B_{i,t}).
	\]
	The set $T' \subseteq T_1 \cup T_2 \subseteq W$ is a subset of some $T'' \in \binom{W}{k}$. We thus arrive at 
	\[
		d_{T'',i}=\dim(U_{T'',i}) \ge \dim(U_{T_1,i}+U_{T_2,i}) > d_i,
	\]
	which contradicts $d_{T'',i} = d_i$. This concludes the proof of the claim.

  	Since $W$ has cardinality $k+1$, we can choose an arbitrary decomposition $W = T \cup \{s\}$, where
  	$T \in \binom{W}{k}$ and  $s \in W \setminus T$. The equalities $\sum_{i=1}^m \sprod{A_{T,i}}{B_{t,i}}=\sprod{A_T}{B_t}=0$ for $t \in T$ and the inequalities $\sprod{A_{T,i}}{B_{t,i}} \ge 0$ for $i \in [m]$ yield $\sprod{A_{T,i}}{B_{t,i}}=0$ for all $t \in T$ and $i \in [m]$. By Lemma~\ref{lem:images},  $\im(A_{T,i})$ is orthogonal to $\im(B_{t,i})$. Since $t \in T$ is arbitrary, we conclude that $\im(A_{T,i})$ is orthogonal to $\sum_{t \in T} \im(B_{t,i}) = U_i$. By the choice of $U_i$, the linear space $U_i$ contains all $\im(B_{w,i})$ with $w \in W$ as a subspace. Hence, $\im(A_{T,i})$ is orthogonal to $\im(B_{s,i})$. By Lemma~\ref{lem:images}, this means that $\sprod{A_{T,i}}{B_{s,i}}=0$ holds for all $i \in[m]$. Thus, we have shown $\sprod{A_T}{B_s}=0$. Since $s \not\in T$, this contradicts the assumptions and yields the desired assertion.
  \end{proof}

 \begin{proof}[Proof of Theorem~\ref{thm:tool}.]
	  It is clear that $v_{n,2d}(x)$, for $x \in X$, belongs to $C^\ast$. Indeed, the inclusion $C \subseteq P_{n,2d}(X)$ implies $0 \le f(x) = \sprod{f}{v_{n,d}(x)}$ for all $f \in C$. We fix an arbitrary $m \in \N$ and show that $C$ has no $(\cS_+^k)^m$-lift. Let 
	  \[
		  N=R_k(k+1,(k+1)^m)
	\] 
	and consider a set $S$ of cardinality $N$ satisfying $(\ast)$. For every $T \in \binom{S}{k}$, choose a polynomial $f_T \in C$ which is equal to $0$ on $T$ and is strictly positive on $S \setminus T$. If $C$ had a $(\cS_+^k)^m$-lift, then by the implication \eqref{K:lift} $\Rightarrow$ \eqref{factorization:general} of Theorem~\ref{gpt:thm}, there would exist $A_T \in (\cS_+^k)^m$ with $T \in \binom{S}{k}$ and $B_s \in (\cS_+^k)^m$ with $s \in S$ such that 
	\[
		f_T(s) = \sprod{f_T}{v_{n,2d}(s)} = \sprod{A_T}{B_s}
	\]
	holds for all $T \in \binom {S}{k}$ and $s \in S$. By construction, $\sprod{A_T}{B_s} =0$ holds if and only if $s \in T$. Thus, assumptions of Lemma~\ref{lem:key} are fulfilled, and Lemma~\ref{lem:key} implies 
	\[
		|S| < R_k(k+1,(k+1)^m)=N,
	\]
	which is a contradiction to $|S|=N$. This shows that $C$ has no $(\cS_+^k)^m$-lift. 
\end{proof}

\section{Proofs of the consequences of Theorem~\ref{thm:tool}} 

\label{sect:proof:conseq:main}

We outline the proof Corollary~\ref{cor:between}. The corollary follows from  Theorem~\ref{thm:tool} by verifying the condition $(\ast)$ for $k=\binom{n+d}{n}-1$ and appropriately chosen polynomials $f$. In view of the assumption $\Sigma_{n,2d} \subseteq C$ in Corollary~\ref{cor:between} one can choose polynomials $f$ in $(\ast)$ to be squares. Thus, we need to find arbitrarily large $S \subseteq X$ such that for every $k$-element subset $T$ of $S$ there exists a polynomial which is equal to zero on $T$ and is not equal to zero on all points of $S \setminus T$. Taking $f$ in $(\ast)$ to be the square of such polynomial, we are able to verify the assumptions of Theorem~\ref{thm:tool} and obtain $\sxd(C) > k$. 

The construction of sets $S$ and the choice of polynomials vanishing on $T$ and not vanishing on $S \setminus T$ relies on the following Lemmas~\ref{lem:generic} and \ref{lem:vanishing}.
We say that a set $V$ of vectors in $\R^n$ is in \emph{general linear position} if every subset of $V$ of cardinality at most $n$ is linearly independent.

\begin{lem}
	\label{lem:generic}
	Let $N \in \N$ and $N \ge \binom{n+d}{n}$. Then the following hold:
	\begin{enumerate}[(a)]
		\item The set of all $(x^1,\ldots,x^N) \in (\R^n)^N$ such that $\{v_{n,d}(x^1),\ldots,v_{n,d}(x^N)\}$ is a set of $N$ vectors in general linear position, is dense in $(\R^n)^N$ in the Euclidean topology.
		\item If $\{x_1,\ldots,x_N\}$ is an $N$-element subset of $\R^n$ such that the set $\{v_{n,d}(x^1),\ldots,v_{n,d}(x^N)\}$ is in general linear position, then every non-zero polynomial $f \in \R[\ux]_d$ is equal to zero on at most $\binom{n+d}{n} -1 $ points of $\{x^1,\ldots,x^N\}$. 
	\end{enumerate}
\end{lem}
\begin{proof}
	(a): Let $k:= \binom{n+d}{n}$.
	For every $I = \{i_1,\ldots,i_k\} \in \binom{[N]}{k}$ with $1 \le i_1 < \ldots \le i_k \le N$, consider 
	\[
		D_I(x^1,\ldots,x^N) := \det ( v_{n,d}(x^{i_1}), \ldots, v_{n,d}(x^{i_k}) ).
	\]
	Clearly, $D_I(x^1,\ldots,x^N)$ can be viewed as a polynomial in $N n$ variables. Note that $D_I(x^1,\ldots,x^N)$ is a non-zero polynomial, as by Leibniz formula, it involves exactly $k!$ distinct monomials. It follows that $D_I(x^1,\ldots,x^N)$ is non-zero on an open dense subset of $(\R^n)^N$. Consequently, all  $D_I(x^1,\ldots,x^N)$ with $I \in \binom{[N]}{k}$ are simultaneously not equal to zero on a dense subset of $(\R^n)^N$. 
	
	(b): Assume that $v_{n,d}(x^1),\ldots,v_{n,d}(x^N)$ are in general position. 
	If $f \in \R[\ux]_d$ vanishes on a $k$-element set $\{x^{i_1},\ldots,x^{i_k} \}$, then
	$f(x^{i_1}) = \cdots = f(x^{i_k})=0$. This condition on $f$, can be viewed as a homogeneous linear system with $k$ equalities in $k$ variables, by interpreting $f$ as an element of $\R^k$. The linearly independent vectors $v_{n,d}(x^{i_1}), \ldots, v_{n,d}(x^{i_k})$ form the left hand side of this system. We thus conclude that $f=0$. 
\end{proof}

\begin{lem}
	\label{lem:vanishing} 
	For every subset $S$ of $\R^n$ of cardinality at most $\binom{n+d}{n} -1$ there exists a non-zero polynomial $f \in \R[\ux]_d$, which is equal to zero on $S$.
\end{lem}
\begin{proof}
	Since the dimension of $\R[\ux]_d$ is $k=\binom{n+d}{n}$ the conditions $f(s)=0$ for all $s \in S$ can be viewed as an under-determined homogeneous linear system in the coefficients of $f$. This implies that there exists a non-zero polynomial $f$ as in the assertion. 
\end{proof}

\begin{proof}[Proof of Corollary~\ref{cor:between}]
	Let $N$ be an arbitrary integer with $N \ge k:=\binom{n+d}{n}-1$.  By Lemma~\ref{lem:generic}, there exist $x^1,\ldots,x^N \in X$ such that $\{v_{n,d}(x^1),\ldots,v_{n,d}(x^N)\}$ is an $N$-element set in general linear position. Let $S = \{x^1,\ldots,x^N\}$. By Lemma~\ref{lem:vanishing}, for every $T \in \binom{S}{k}$ there exists a non-zero polynomial $f_T \in \R[\ux]_d$ equal to zero on $T$. In view of Lemma~\ref{lem:generic}, $f_T$ is not equal to zero on $S \setminus T$. The square $f_T^2$ of $f_T$ belongs to $\Sigma_{n,2d}$ and by this also to $C$. Since $N$ is chosen arbitrarily, assumptions of Theorem~\ref{thm:tool} are fulfilled. We thus conclude that $\sxd(C) > k$. 
\end{proof}

\begin{rem}
	In the case $n=1$, in the above proof, one could also choose $x^1,\ldots,x^N \in X$ to be arbitrary distinct values and fix
	\[
	f_T := \prod_{t \in T} (x-t) \in \R[x]
	\]
	This was also the choice used for deriving the lower bound $\sxd(\Sigma_{1,4}) \ge 3$ in  \cite[Sect.~IV-B]{ahmadi2017improving}.
\end{rem}

\begin{rem}
	In the case $d=1$, in the above proof, $v_{n,1}(x) \in \R^{n+1}$ is obtained from $x \in \R^n$ by appending a component $1$. So, one can choose
	\[
		x^i  = (x^\ast + t_i^1,\ldots, x^\ast + t_i^n)
	\]
	using a point  $x^\ast$ in the interior of $X$ and $N$ distinct values $t_1,\ldots,t_N \in \R$ that are sufficiently close to zero. With this choice, the set $\{v_{n,1}(x^1),\ldots,v_{n,1}(x^N)\}$ of $N$ vectors is in general linear position.  The latter can be seen by observing that $\det(v_{n,1}(x^{i_1}),\ldots,v_{n,1}(x^{i_{n+1}}))$ is the Vandermonde determinant.
\end{rem}

Applying Corollary~\ref{cor:between} in the case $C=\Sigma_{n,2d}$, we determine $\sxd(\Sigma_{n,2d})$ and $\sxc(\Sigma_{n,2d})$:

\begin{proof}[Proof of Corollary~\ref{cor:sxd:sos}] As mentioned in the introduction, the inequalities 
	\[
		\sxd(\Sigma_{n,2d}) \le \sxc(P_{n,2d}) \le \binom{n+d}{n}
	\] are known. Applying Corollary~\ref{cor:between} for $C=\Sigma_{n,2d}$, we get $\sxd(\Sigma_{n,2d}) \ge \binom{n+d}{n}$. 
\end{proof}

Truncated quadratic modules are sums of finitely many cones. To determine the semidefinite extension degree of the truncated quadratic modules, we first make an observation on how the semidefinite extension degree behaves with respect to taking sums: 

\begin{lem}
	\label{lem:sxd:sum}
	Let $C_1,\ldots,C_k \subseteq \R^n$. Then 
	\[
	\sxd(C_1 + \cdots + C_k) \le \max \setcond{\sxd(C_i)}{i =1,\ldots,k}. 
	\]
\end{lem}
\begin{proof}
	This follows directly from the fact that $C_1 + \cdots + C_k$ is a linear image of $C_1 \times \cdots \times C_k$ under the linear map $(u_1,\ldots,u_k) \mapsto u_1 + \cdots + u_k$ acting from $(\R^n)^k$ to $\R^n$. 
\end{proof}

As we will see in the proof of Corollary~\ref{cor:sxd:trunc:quad:module}, for cones occurring in the definition of the truncated quadratic module and under assumptions of Corollary~\ref{cor:sxd:trunc:quad:module}, the inequality in Lemma~\ref{lem:sxd:sum} is in fact an equality. The proof of Corollary~\ref{cor:sxd:trunc:quad:module} reuses the proof approach of Corollary~\ref{cor:between}.

\begin{proof}[Proof of Corollary~\ref{cor:sxd:trunc:quad:module}]
	Let $g_0:=1$ and $C_i := g_i \Sigma_{n,2 d_i}$ for $i=0,\ldots,d$. In this notation, one has $C = C_0 + \cdots + C_k$. Since the cone $C_i$ is linearly isomorphic to $\Sigma_{n,2 d_i}$ for each $i =0,\ldots,k$ and since the semidefinite extension degree of $\Sigma_{n,2d_i}$ is determined by Corollary~\ref{cor:sxd:sos}, taking into account Lemma~\ref{lem:sxd:sum}, we  obtain the upper bound
	\begin{align*}
		\sxd(C) & \le \max \setcond{ \sxd(C_i) }{i=0,\ldots,k} 
		\\ & = \max \setcond{ \sxd(\Sigma_{n,2 d_i}) }{i=0,\ldots,k}
		\\ & = \binom{n+d}{n}
	\end{align*}
	on $\sxd(C)$. We can adapt the proof of Corollary~\ref{cor:between} to derive the matching lower bound 
	\(
		\sxd(C) \ge \binom{n+d}{n}.
	\)
	Fix $i=0,\ldots,k$ with $d_i=d$. Since $g_i$ is not a zero polynomial, the set $\setcond{x \in X}{g_i(x) \ne 0}$ is $n$-dimensional. Thus, we can fix arbitrarily many points $x^1,\ldots,x^N$ with $N \ge \binom{n+d}{n} -1$ as in the proof of Corollary~\ref{cor:between} that satisfy the additional assumption $g_i(x^j) \ne 0$ for $j \in [N]$.  The polynomials $f_T$ with $T \subseteq S$ and $|T|=\binom{n+d}{n}-1$ from the proof of Corollary~\ref{cor:between} give rise to polynomials $g_i f_T^2$ in $C_i$ that vanish on $T$ and are strictly positive on $S \setminus T$. In view of Theorem~\ref{thm:tool}, we get $\sxd(C) \ge \binom{n+d}{n}$. 
\end{proof}

To prove Corollary~\ref{cor:sxd:sdp}, it suffices to observe that $\cS_+^k$ is linearly isomorphic to $\Sigma_{k-1,2}$:
\begin{proof}[Proof of Corollary~\ref{cor:sxd:sdp}]
We assume $k \ge 2$, as otherwise the assertion is trivial. 
Consider the maps \(A \mapsto q_A \mapsto f_A\) given by
\begin{align*}
	q_A(x) & := x^\top A x,
	\\ f_A(x_1,\ldots,x_{k-1}) &:= q_A(x_1,\ldots,x_{k-1},1).
\end{align*}
It is straightforward to see that $A \mapsto f_A$ is a linear bijection acting from $\cS^k$ to $\R[x_1,\ldots,x_{k-1}]_2$ that maps $\cS_+^k$ onto $\Sigma_{k-1,2}$. Thus, the assertion follows by applying Corollary~\ref{cor:sxd:sos} for $d=1$ and $n=k-1$. 
\end{proof}

The cases of equality $P_{n,2d}=\Sigma_{n,2d}$ are characterized by a classical result of Hilbert, while Scheiderer's result result shows that, in the case $P_{n,2d} \ne \Sigma_{n,2d}$, the cone $P_{n,2d}$ has no semidefinite extended formulation. In Corollary~\ref{cor:nonneg:cone}, we use these results and the knowledge of $\sxd(\Sigma_{n,2d})$ to determine $\sxd(P_{n,2d})$.

\begin{proof}[Proof of  Corollary~\ref{cor:nonneg:cone}]
	If $n,d \ge 2$ and $n,d \ne (2,2)$, Theorem~\ref{thm:scheiderer} of Scheiderer implies $\sxd(P_{n,2d})=\sxc(P_{n,2d})=\infty$. Otherwise, by Theorem~\ref{thm:hilbert} of Hilbert, $P_{n,2d}=\Sigma_{n,2d}$. Thus, $\sxd(P_{n,2d})=\sxc(P_{n,2d})=\binom{n+d}{n}$ follows using Corollary~\ref{cor:sxd:sos}.
\end{proof}

Corollaries~\ref{cor:mom:cones:X} and \ref{cor:mom:cones} are obtained through straightforward dualization. Moment cones are known to be dual to cones of non-negative polynomials:

\begin{lem}[Folklore]
	\label{dual:nonneg}
	$P_{n,2d}(X)^\ast = M_{n,2d}(X)$ for every $X \subseteq \R^n$. In particular, $P_{n,2d}^\ast = M_{n,2d}$. 
\end{lem}
\begin{proof}
	Writing evaluation of $f$ at $x \in \R^n$ as the scalar product $f(x) = \sprod{f}{v_{n,2d}(x)}$, we obtain 
	$P_{n,2d}(X) = (\setcond{v_{n,2d}(x)}{x \in X})^\ast$. Dualizing the latter equation and using \eqref{double:dual}, we get $P_{n,2d}(X)^\ast = M_{n,2d}(X)$. 
\end{proof}

\begin{proof}[Proof of Corollary~\ref{cor:mom:cones:X}]
	By Lemma~\ref{dual:nonneg}, $\sxd(M_{n,2d}(X))=\sxd(P_{n,2d}(X)^\ast)$.
	By \eqref{inv:dual:sxd},  the semidefinite extension degree is preserved under duality, so that one has $\sxd(P_{n,2d}(X)^\ast)=\sxd(P_{n,2d}(X))$.
	By Corollary~\ref{cor:between}, $\sxd(P_{n,2d}(X)) \ge \binom{n+d}{n}$. This yields the assertion. 
\end{proof}

\begin{proof}[Proof of Corollary~\ref{cor:mom:cones}]
	We have $M_{n,2d} = P_{n,2d}^\ast$, by Lemma~\ref{dual:nonneg}. By \eqref{inv:dual:sxc} and \eqref{inv:dual:sxd} the semidefinite extension degree and the semidefinite extension complexity are preserved under duality, so that one has $\sxd(P_{n,2d}^\ast)=\sxd(P_{n,2d})$ and $\sxc(P_{n,2d}^\ast)=\sxc(P_{n,2d})$. We conclude that $M_{n,2d}$ has the same semidefinite extension degree and the semidefinite extension complexity as the cone $P_{n,2d}$. An application of Corollary~\ref{cor:nonneg:cone} concludes the proof.
\end{proof}

The lower bound on the semidefinite extension degree of the copositive cone  $\CP_k$ in Corollary~\ref{cor:copositive} is established by interpreting $\CP_k$ as $P_{n,2d}(X)$ with an appropriate choice of $n,d$ and $X$. The matching upper bound for $k \le 4$ is a direct consequence of the fact that, for $k \le 4$, a symmetric $k \times k$ matrix is copositive if and only if it is a sum of a non-negative matrix and a positive semidefinite matrix \cite{MaxfieldMinc1962}. 

\begin{proof}[Proof of Corollary~\ref{cor:copositive}]
	We assume  $k \ge 2$ to exclude the trivial case $k=1$. 
	We can use the linear bijection $A \mapsto f_A$ from the proof of  Corollary~\ref{cor:sxd:sdp}. It is easy to see that this bijection sends $\CP_k$ onto $P_{k-1,2}(\R_+^{k-1})$. Thus, by Theorem~\ref{cor:between}, we get $\sxd(\CP_k)=\sxd(P_{k-1,2}(\R_+^{k-1})) \ge k$. 
	
	It is known that $\CP_k = \cS_+^k + \mathcal{N}_+^k$ holds for $k \le 4$, where $\mathcal{N}_+^k := \cS^k \cap \R_+^{k \times k}$ is the cone of symmetric $k \times k$ matrices with non-negative components (see \cite{MaxfieldMinc1962} and \cite[Sect.~3]{Duer:2010}). Lemma~\ref{lem:sxd:sum} yields $\sxd(\CP_k) \le \max \{ \sxd(\cS_+^k), \sxd(N_+^k) \}  = k.$ 
\end{proof}

\section{Proofs of results for the SONC cone}

\label{sec:SONC}


In this section, we first convert the existing description of the SONC cone by Iliman and Timo de Wolff to an alternative description, which is more convenient for our purposes. Once the alternative description is obtained, the existence of a second-order cone extended formulation for the SONC cone, will follow from a well-known result of Ben-Tal and Nemirovski. 


The definition of the SONC cone involves non-negative circuit polynomials: 

\begin{dfn}[Non-negative circuit polynomials and circuit number \cite{MR3481195,MR3691721}]
	\label{def:ncp}
	Let $\cA_{n}$ be the set of all $A \subseteq \Z_+^n$ of the form $A=\{\alpha(0),\ldots,\alpha(k),\beta\}$, where $k \in [n]$, with the following properties:
\begin{enumerate}
	\item $\alpha(0),\ldots,\alpha(k) \in (2 \Z_+)^n$,
	\item $\alpha(0),\ldots,\alpha(k)$ are vertices of a $k$-dimensional simplex,
	\item $\beta$ is in the relative interior of the simplex with the vertices $\alpha(0),\ldots,\alpha(k)$, that is, 
	\begin{align*}
		\beta & = \sum_{i=0}^k \lambda_i \alpha(i)
		& & \text{and} &  1 & = \sum_{i=0}^k \lambda_i
	\end{align*}
	holds for some coefficients $\lambda_0 > 0,\ldots,\lambda_k > 0$ uniquely determined by $\alpha(0),\ldots,\alpha(k)$ and $\beta$. 
\end{enumerate}
Elements of the set
\[
	\tilde{P}_{n,A} := \setcond{f = \sum_{\alpha \in A} f_\alpha \ux^\alpha}{f \ge 0 \ \text{on} \ \R^n, \ f_{\alpha(0)} > 0,\ldots, f_{\alpha(k)} > 0}.
\]
are called \emph{non-negative circuit polynomials} with respect to the circuit $A \in \cA_n$. If $f = \sum_{\alpha \in A} f_\alpha \ux^\alpha$, is polynomial satisfying $f_{\alpha(i)} \ge 0$ for all $i \in \{0,\ldots,k\}$, then the value
\begin{equation}
\label{Theta:f}
\Theta_f := \prod_{i=0}^k \left( \frac{f_{\alpha(i)}}{\lambda_i} \right)^{\lambda_i}.
\end{equation}
is called the \emph{circuit number} of $f$.
\end{dfn}

\begin{thm}[Iliman and Timo de Wolff {\cite[Theorem~3.8]{MR3481195}}]
	\label{thm:idw}
In the notation of Definition~\ref{def:ncp}, the set $\tilde{P}_{n,A}$ is described as 
\begin{align*}
\tilde{P}_{n,A} & = \setcond{f = \sum_{\alpha \in A} f_\alpha \ux^\alpha}{f_{\alpha(0)} > 0, \ldots, f_{\alpha(k)} > 0, \ f_\beta \ge -\Theta_f}  & & \text{if} \ \beta \in (2 \Z_+)^n
\end{align*}
and 
\begin{align*}
\tilde{P}_{n,A} & = \setcond{f = \sum_{\alpha \in A} f_\alpha \ux^\alpha}{f_{\alpha(0)} > 0, \ldots, f_{\alpha(k)} > 0, \ \Theta_f \ge f_\beta \ge -\Theta_f}  & & \text{if} \ \beta \not\in (2 \Z_+)^n.
\end{align*}	
\end{thm}

While $\tilde{P}_{n,A}$ has a nice explicit description, it has a minor technical drawback of being neither open, nor closed nor a convex cone ($\tilde{P}_{n,A}$ is missing the zero polynomial for being a convex cone). Essentially, the cone $P_{n,A}$ of non-negative polynomials whose support is a subset of $A$, defined by equality \eqref{PnA:eq} in the introduction, is a `regular' version of $\tilde{P}_{n,A}$ with a completely analogous description: 

\begin{lem}
	\label{lem:circuit}
	In the notation of Definition~\ref{def:ncp}, the set $P_{n,A}$ is described as 
	\begin{align*}
		P_{n,A} & = \setcond{f = \sum_{\alpha \in A} f_\alpha \ux^\alpha}{f_{\alpha(0)} \ge 0, \ldots, f_{\alpha(k)} \ge 0, \ f_\beta \ge -\Theta_f}  & & \text{if} \ \beta \in (2 \Z_+)^n
	\end{align*}
	and 
	\begin{align*}
		P_{n,A} & = \setcond{f = \sum_{\alpha \in A} f_\alpha \ux^\alpha}{f_{\alpha(0)} \ge 0, \ldots, f_{\alpha(k)} \ge 0, \ \Theta_f \ge f_\beta \ge -\Theta_f}  & & \text{if} \ \beta \not\in (2 \Z_+)^n.
	\end{align*}	
\end{lem}
\begin{proof}
	First observe that, if $f$ is in $P_{n,A}$ then $f_{\alpha(i)} \ge 0$ holds for every $i=0,\ldots,k$. Let us show $f_{\alpha(0)} \ge 0$. Choose a vector $\gamma = (\gamma_1,\ldots,\gamma_n) \in \Z^n \setminus \{0\}$ such that 
	\[
		\sprod{\gamma}{\alpha(0)} > \sprod{\gamma}{\beta} > \sprod{\gamma}{\alpha(1)}=\cdots =\sprod{\gamma}{\alpha(k)}.
	\] 
	The vector $\gamma$ is an inner facet normal of the simplex with the vertices $\alpha(0),\ldots,\alpha(k)$.
	For $t \in \R_+$, we obtain 
	\[
		q(t):=f(t^{\gamma_1},\ldots,t^{\gamma_n}) = f_{\alpha(0)} t^{\sprod{\alpha(0)}{\gamma}} + f_\beta t^{\sprod{\beta}{\gamma}} + \left( \sum_{i=1}^k f_{\alpha(i)} \right) t^{\sprod{\alpha(1)}{\gamma}}.
	\]
	If one had $f_{\alpha(0)} < 0$, then $q(t)$ would be negative for a sufficiently large $t$. Hence $\alpha(0) \ge 0$, and analogously we obtain $f_{\alpha(i)} \ge 0$ for every $i=0,\ldots,k$.

	We introduce the $\epsilon$-perturbation of $f$ by 
	\begin{align*}
		f_\epsilon & := \sum_{i=0}^k (f_{\alpha(i)} + \epsilon) \ux^{\alpha(i)} + f_\beta \ux^\beta = f + \epsilon \sum_{i=0}^k \ux^{\alpha(i)}.
	\end{align*}
	Since $\alpha(0),\ldots,\alpha(k) \in (2 \Z_+)^n$, the non-negativity of $f$ implies the non-negativity of $f_\epsilon$ for every $\epsilon>0$. Thus, if $f \in P_{n,A}$, then $f_\epsilon$ is non-negative for every $\epsilon>0$. Since $f_\epsilon \in \tilde{P}_{n,A}$, applying a description of $\tilde{P}_{n,A}$ from Theorem~\ref{thm:idw}, and letting $\epsilon>0$ go to $0$, we derive the `$\subseteq$' parts of the equalities of our assertion.

	Conversely, if $\beta \in (2 \Z_+)^n$ and the inequalities $f_{\alpha(0)} \ge 0,\ldots, f_{\alpha(k)} \ge 0, f_\beta \ge - \Theta_f$ are fulfilled, then in the case $f_{\alpha(0)} > 0,\ldots,f_{\alpha(k)}>0$, one has $f \in \tilde{P}_{n,A} \subseteq P_{n,A}$, while in the case $f_{\alpha(i)}=0$ for some $i=0,\ldots,k$ one has $\Theta_f=0$ so that $f$ is a sum of squares of $k+2$ monomial terms. 
	
	Similarly, if $\beta \not\in (2 \Z_+^n)$, then carrying out the same case distinction, we conclude that one has $f \in \tilde{P}_{n,A} \subseteq P_{n,A}$ or, otherwise, $f$ is a sum of squares of $k+1$ monomial terms.
\end{proof}

Comparing Lemma~\ref{lem:circuit} and Theorem~\ref{thm:idw}, we see that every polynomial $f$ in $P_{n,A}$ is either in $\Tilde{P}_{n,A}$ or is a non-negative linear combination of squares of monomials. 

\begin{dfn}[SONC polynomials \cite{MR3481195,MR3691721}]
	\label{def:C:n:d}
	Let $n, d \in \N$. Using the set $\cA_{n}$ from Definition~\ref{def:ncp}, we define 
	\[
		\cA_{n,2d} := \setcond{A \in \cA_n}{|\alpha| \le 2d \ \text{for all} \ \alpha \in A}
	\]
	Let $C_{n,2d}$ be the set of all polynomials $f \in \R[\ux]_{2d}$ that can be written as 
	\[
		f = \mu_1 f_1 + \cdots + \mu_N f_N,
	\]
	where $N \in \N$, $\mu_1,\ldots,\mu_N \ge 0$, and, for every $i \in \{1,\ldots, N\}$, the polynomial $f_i$ is 
	\begin{enumerate}
		\item \label{circ:pol} either an element of $\Tilde{P}_{n,A}$ for some $A \in \cA_{n,2d}$ 
		\item \label{mon:square} or a square $f_i = \ux^{2 \alpha}$ of some monomial $\ux^\alpha$ with $\alpha \in \Z_+^n$ and $|\alpha| \le d$. 
	\end{enumerate}
	
	Polynomials from $C_{n,2d}$ are called \emph{sums of non-negative circuit (SONC) polynomials} of degree at most $2d$ in $n$ variables. 
\end{dfn}

\begin{rem}
	From definitions given in 
	\cite{MR3481195,MR3691721} it is not immediately clear if monomial squares $x^{2 \alpha}$ are supposed to be SONC polynomials.  According to explanations given by Timo~De~Wolff~\cite{de:wolff:private}, the authors of \cite{MR3481195,MR3691721} did intend to view monomial squares as degenerate SONC polynomials. In Definition~\ref{def:C:n:d}, the `shape' of the cone $C_{n,2d}$ is determined by Condition~\ref{circ:pol}, while adding monomial squares via Condition~\ref{mon:square} makes the cone $C_{n,2d}$ topologically closed. 
\end{rem}

\begin{rem}
In view of Theorem~\ref{thm:idw} and Lemma~\ref{lem:circuit}, every element of $P_{n,A} \setminus \Tilde{P}_{n,A}$, with $A \in \cA_n$, is a conic combination of monomial squares. This shows that, for $n,d \in \N$, the SONC cone $C_{n,2d}$ can be represented as 
\begin{equation}
	\label{reg:SONC:sum}
	C_{n,2d} = \sum_{A \in \cA_{n,2d}} P_{n,A}.
\end{equation}
Equality \eqref{reg:SONC:sum} is a non-technical alternative definition of $C_{n,2d}$.
\end{rem}

We can easily determine the semidefinite extension degree of $C_{n,2d}$ from \eqref{reg:SONC:sum}, as $\sxd(P_{n,A})$ for $A \in \cA_{n,2d}$ can be calculated using the following result from \cite{MR1857264}.

\begin{lem}[Ben-Tal and Nemirovski {\cite[Example~15 in \S\S2.3.5]{MR1857264}}]
	\label{lem:BT:N}
	Let $\lambda_1,\ldots,\lambda_m$ be positive rational numbers satisfying $\lambda_1 + \cdots + \lambda_m \le 1$. Then, for the set
	\[
		C := \setcond{(x_1,\ldots,x_m,x_{m+1}) \in \R_+^m \times \R}{x_{m+1} \le x_1^{\lambda_1} \cdots x_m^{\lambda_m}},
	\]
	one has $\sxd(C) \le 2$.
\end{lem}
\begin{proof}
	Constructions in \cite[\S\S2.3.5]{MR1857264} yield an explicit second-order cone extended formulation of $C$, which shows $\sxd(C) \le 2$. 
\end{proof}

\begin{proof}[Proof of Theorem~\ref{sxd:SONC}]
	Equality \eqref{reg:SONC:sum} describes $C_{n,2d}$ as a sum of finitely many closed convex cones. Hence, applying Lemma~\ref{lem:sxd:sum} to the cone $C_{n,2d}$ represented by \eqref{reg:SONC:sum}, we obtain 
	\[
		\sxd(C_{n,2d}) \le \max \setcond{P_{n,A}}{A \in \cA_{n,2d}}.
	\]
	The description of cones $P_{n,A}$ given in Lemma~\ref{lem:sxd:sum} is in terms of non-strict linear inequalities $f_{\alpha(0)} \ge 0,\ldots, f_{\alpha(k)} \ge 0$ and the inequalities which coincide, up to a rescaling of the variables, with the inequalities describing the set $C$ in Lemma~\ref{lem:BT:N}. Thus, Lemma~\ref{lem:BT:N} yields $\sxd(P_{n,A}) \le 2$ for every $A \in \cA_{n,2d}$ and we obtain $\sxc(C_{n,2d}) \le 2$. 
	
	It remains to show $\sxd(C_{n,2d}) \ge 2$, which means that $C_{n,2d}$ is not a polyhedron. One way to see this is to use Theorem~\ref{thm:tool} in the degenerate case $k=1$. Take $S \subseteq \R^n$ to be an arbitrarily large finite subset of the $x_1$-axis $\R \times \{0\}^{n-1}$. For each $s = (s_1,0,\ldots,0) \in S$, the quadratic polynomial $f = (x_1 - s_1)^2 \in \R[\ux]$ belongs to $C_{n,2d}$ and is equal to zero on exactly one point of $S$. So, by Theorem~\ref{thm:tool},  $\sxd(C_{n,2d}) > 1$. 
\end{proof}

\begin{proof}[Proof of Corollary~\ref{cor:SOS+SONC}]
	If $n=1$ or $d=1$ or $(n,d)=(2,2)$, then $\Sigma_{n,2d} = P_{n,2d}$ by Theorem~\ref{thm:hilbert}. Hence $\Sigma_{n,2d} + C_{n,2d} = P_{n,2d}$. 
		
	In the case $n,d \ge 2$ and $(n,d) \ne (2,2)$, Theorem~\ref{thm:scheiderer} of Scheiderer  asserts that  $P_{n,2d}$ has no semidefinite extended formulation. On the other hand, in view Lemma~\ref{lem:sxd:sum}, the cone $\Sigma_{n,2d} + C_{n,2d}$ \emph{does} have a semidefinite extended formulation since both summands $\Sigma_{n,2d}$ and $C_{n,2d}$ have a semidefinite extended formulation, by Corollary~\ref{cor:sxd:sos} and Theorem~\ref{sxd:SONC}, respectively. This implies $\Sigma_{n,2d} + C_{n,2d} \ne P_{n,2d}$.
\end{proof}

\subsection*{Acknowledgements} 

I thank Jonas Frede for pointing to \cite{ahmadi2017improving}. Theorem~\ref{sxd:SONC} and Corollary~\ref{cor:SOS+SONC} were motivated by the discussion with Andreas Bernig, Mareike Dressler, Raman Sanyal and Thorsten Theobald during the defense of Mareike's PhD thesis \cite{Dressler:Diss}. I would like to thank Timo de Wolff for clarifying the definition of the SONC cone. 

\subsection*{Funding}

The project is funded by the Deutsche Forschungsgemeinschaft (DFG, German Research Foundation) - 314838170, GRK 2297 MathCoRe.

\bibliographystyle{amsalpha}
\bibliography{literature}

\end{document}